\documentclass[10pt]{amsart}

\usepackage{hyperref}
\hypersetup{colorlinks=true, urlcolor=blue, citecolor=blue, linkcolor=blue}

\usepackage{amsmath,amsthm, amssymb}

\usepackage{graphicx,mathrsfs,tikz,fancyhdr}

\usepackage{mdwlist}

\usepackage{thmtools}
\usepackage{nameref}
\usepackage[backend=bibtex]{biblatex}

\bibliography{Masseybib.bib}

\newcommand{\U}{{\mathcal U}}
\newcommand{\0}{{\mathbf 0}}
\newcommand{\C}{{\mathbb C}}
\newcommand{\Z}{{\mathbb Z}}

\newcommand{\cL}{{\mathbb L}}
\newcommand{\D}{{\mathbb D}}
\newcommand{\W}{{\mathcal W}}

\newcommand{\strat}{{\mathfrak S}}
\newcommand{\Proj}{{\mathbb P}}
\newcommand{\hyp}{{\mathbb H}}

\newcommand{\im}{\mathop{\rm im}\nolimits}

\newcommand{\Fdot}{\mathbf F^\bullet}

\newtheorem{defn0}{Definition}[section]
\newtheorem{prop0}[defn0]{Proposition}
\newtheorem{conj0}[defn0]{Conjecture}
\newtheorem{thm0}[defn0]{Theorem}
\newtheorem{lem0}[defn0]{Lemma}
\newtheorem{corollary0}[defn0]{Corollary}
\newtheorem{example0}[defn0]{Example}
\newtheorem{remark0}[defn0]{Remark}
\newtheorem{question0}[defn0]{Question}

\newenvironment{defn}{\begin{defn0}}{\end{defn0}}
\newenvironment{prop}{\begin{prop0}}{\end{prop0}}

\newenvironment{thm}{\begin{thm0}}{\end{thm0}}

\newenvironment{cor}{\begin{corollary0}}{\end{corollary0}}
\newenvironment{exm}{\begin{example0}\rm}{\end{example0}}
\newenvironment{rem}{\begin{remark0}\rm}{\end{remark0}}

\newcommand{\defref}[1]{Definition~\ref{#1}}
\newcommand{\propref}[1]{Proposition~\ref{#1}}
\newcommand{\thmref}[1]{Theorem~\ref{#1}}

\newcommand{\exref}[1]{Example~\ref{#1}}

\newcommand{\remref}[1]{Remark~\ref{#1}}

\title[Affine IPA-Deformations]{IPA-deformations of functions on affine space}

\author{David B. Massey}
\address{Department of Mathematics\\
  Northeastern University\\
  Boston, Massachusetts 02115}
\email[D.~Massey]{d.massey@neu.edu}
\subjclass[2010]{32S05, 32S60, 32S30, 32S55, 32S50, 32S15}

\begin{document}

\begin{abstract}
We investigate deformations of functions on affine space, deformations in which the changes specialize to a distinguished point in the zero-locus of the original function. Such deformations enable us to obtain nice results on the cohomology of the Milnor fiber of the original function.
\end{abstract}

\maketitle

\thispagestyle{fancy}

\lhead{}
\chead{}
\rhead{ }

\lfoot{}
\cfoot{}
\rfoot{}

\section{Introduction} 

It is standard to investigate singularities of maps and spaces by looking at deformations. In our 20-year-old paper \cite{prepolardef}, we defined {\it prepolar deformations}, but then did very little with the notion.

In this paper, we will define a notion -- {\it deformations with isolated polar activity or IPA-deformations} -- which is a more general notion than that of our old prepolar deformations. An IPA-deformation is roughly analogous to an unfolding of a map-germ such that the map-germ has an isolated instability with respect to the given unfolding.  

We will use many of our results from \cite{enrichpolar}, in which we used the derived category, nearby cycles, vanishing cycles, and a generalized notion of the relative polar curve. We did this with respect to an arbitrary bounded, constructible complex of $\Z$-modules on an arbitrary complex analytic space (and the results hold for more general base rings which are commutative, regular, Noetherian, with finite Krull dimension). However, the extreme generality of \cite{enrichpolar} makes the results there almost incomprehensible.

Even in the well-studied case of functions on affine space, the results we obtain are  non-trivial, but hard to decipher from \cite{enrichpolar}. Some of the results that we obtain are familiar, but with weaker hypotheses, as in \thmref{thm:vangen}, while the main result for IPA-deformations, \thmref{thm:ipa}, is new, and unavoidably involves hypercohomology with coefficients in the sheaf of vanishing cycles.

The basic essence of our results will certainly not be surprising for experts; the results say that if one has a complex analytic function germ at the origin in affine space and one deforms the function in such a way that the origin is the ``only place where the deformation changes the function'', then the Milnor fiber of the original function at the origin has cohomology which ``changes only in the top possible degree'' (that is, in middle dimension). Of course, making this precise is not simple.

\medskip

In Section 2, we give the ``correct'' definition of the relative polar curve; this is a variant of  the classical relative polar curve of Hamm, L\^e, and Teissier (see, for instance,  \cite{hammlezariski}, \cite{teissiercargese}, \cite{leattach}, \cite{letopuse}). This definition allows us to easily specify the genericity that we need for our deformation parameter in order to obtain our results. We prove a number of properties of this generalized relative polar curve.

In Section 3, we recover three classical results on Milnor fibers and complex links, but we use weaker hypotheses; on the other hand, our results are weaker, as they are on the level of cohomology, not homotopy-type.

In Section 4, we define an IPA-deformation $f$ of a function $f_0$  and give the main theorem, \thmref{thm:ipa}. This theorem gives the reduced cohomology of the Milnor fiber $F_{f_0, \0}$ in terms of a single intersection number and the hypercohomology $\hyp^k(F_{t,\0}\cap\Sigma f; \phi_f\Z_\U^\bullet)$, where $F_{t, \0}\cap \Sigma f$ is the Milnor fiber of the deformation parameter $t$ restricted to the critical locus $\Sigma f$ and $\phi_f\Z_\U^\bullet$ is the vanishing cycles of the constant sheaf along $f$. We give two examples, where $f_0$ has a $0$- and $1$-dimensional critical locus, for which we can explicitly calculate $\widetilde H(F_{f_0, \0};\Z)$ via \thmref{thm:ipa}.

\medskip

\section{The Relative Conormal Space and Generalized Polar Curve}\label{sec:rel}

Let $\U$ be a connected open neighborhood of the origin in $\C^{n+1}$ and let $f:(\U,\0)\rightarrow (\C,0)$ be a complex analytic function which is not identically zero. We let $t$ be a non-zero linear form on $\C^{n+1}$, restricted to $\U$. We write $\Sigma f$ (respectively, $\Sigma(f,t)$) for the critical locus of $f$ (respectively, of  $(f,t)$). We assume that $\U$ is chosen small enough so that $\Sigma f\subseteq V(f):=f^{-1}(0)$.
\medskip

Suppose that $M$ is a complex submanifold of $\U$. For $x\in M$, let $T_xM$ denote the tangent space of $M$ at $x$. Let $T^*\U\cong \U\times\C^{n+1}$ denote the cotangent space of $\U$ and let $\pi:T^*\U\rightarrow \U$ be the projection.

 Recall the following now-classic definitions (see, for instance, \cite{teissiervp2}).

\begin{defn}\label{def:main} 

The {\bf conormal space} $T^*_M\U$ is defined by
$$
T^*_M\U:= \{(x, \eta)\in T^*\U\cap \pi^{-1}(M)\ |\ \eta(T_xM)=0\}.
$$

The {\bf relative conormal space} $T^*_{f_{|_M}}\U$ is defined by 
$$
T^*_{f_{|_M}}\U :=\{(x, \eta)\in T^*\U\cap \pi^{-1}(M)\ |\ \eta(T_xM\cap \ker d_x f)=0\}.
$$

\end{defn}

\smallskip

\begin{rem} We have two comments about the definition of the relative conormal space.

\smallskip

\begin{itemize}

\item The fiber, $(T^*_f\U)_x$, of $T^*_f\U$ over a point $x\in \U$ is easy to describe:

$$
(T^*_f\U)_x \ = \ \begin{cases}\{\lambda\,d_xf\}_{\lambda\in\C}, &\textnormal{ if } x\not\in \Sigma f;\\
0, &\textnormal{ if } x\in \Sigma f.
\end{cases}
$$

\medskip

\item Note that there is an equality of the closures in $T^*\U$ given by
$$
\overline{T^*_{f_{|_{\U-\Sigma f}}}\U}  \ =  \ \overline{T^*_f\U},
$$
and this analytic set is irreducible of dimension $n+2$.

\end{itemize}
\end{rem}

\medskip

We shall need the following proposition later.

\begin{prop}\label{prop:contain} Suppose that $Y$ is an irreducible component of $\Sigma f$.  Then, $\overline{T^*_{Y_{\operatorname{reg}}}\U}\subseteq \overline{T^*_f\U}$. In fact, for an open dense (in the analytic Zariski topology) subset $Q\subseteq Y_{\operatorname{reg}}$, there is an equality 
$$T^*_Q\U= \overline{T^*_f\U}\cap \pi^{-1}(Q).$$

\end{prop}
\begin{proof} This is highly non-trivial, and we must use techniques and results from \cite{singenrich}. Recall that we are assuming that $\U$ is chosen small enough so that $\Sigma f\subseteq V(f)$. 
\smallskip

Since $\Z^\bullet_{V(f)}[n]$ is perverse, there is a canonical short exact sequence of perverse sheaves
$$
0\rightarrow \Z^\bullet_{V(f)}[n]\rightarrow \psi_f[-1]\Z^\bullet_\U[n+1]\rightarrow \phi_f[-1]\Z^\bullet_\U[n+1]\rightarrow 0,
$$
which implies that there is a containment of micro-supports 
\begin{equation}
SS\big(\phi_f[-1]\Z^\bullet_\U[n+1]\big)\subseteq SS\big(\psi_f[-1]\Z^\bullet_\U[n+1]\big).\tag{$\dagger$}
\end{equation}
Theorem 3.3 of \cite{singenrich} tells us that
$$
SS\big(\psi_f[-1]\Z^\bullet_\U[n+1]\big)=\overline{T^*_f\U}\cap (f\circ\pi)^{-1}(0).
$$
We also know that the support of $\phi_f[-1]\Z^\bullet_\U[n+1]$ is $\Sigma f$, which implies that, for all irreducible components $Y$ of $\Sigma f$, 
$$\overline{T^*_{Y_{\operatorname{reg}}}\U}\subseteq SS\big(\phi_f[-1]\Z^\bullet_\U[n+1]\big).
$$
The first conclusion now follows from ($\dagger$).

The final statement follows from the former statement, together with the fact that there exists an $a_f$ stratification of $V(f)$ (see  \cite{hironakastratflat}). This implies that there is an open dense subset $S$ of $Y_{\operatorname{reg}}$, a stratum, such that 
$$
\overline{T^*_f\U}\cap\pi^{-1}(S)\subseteq T^*_S\U,
$$
which finishes the proof.
\end{proof}

\medskip

Below, we consider the intersection product, $(- \cdot -)$; as we will always be dealing with proper intersections inside a smooth manifold, this will yield a well-defined intersection cycle ({\bf not} just a cycle class modulo rational equivalence). See \cite{fulton}, Section 8.2 and Example 11.4.4.

We also consider $\im dt$, the image of the differential of $t$. To be clear, this means that
$$
\im dt :=\{(x, d_xt)\in T^*\U \ | \ x\in\U\};
$$
Note that $\im dt$ has dimension $n+1$, so that $\overline{T^*_f\U}$ properly intersects $\im dt$ in $T^*\U$ if and only if 
$$
\dim \left(\overline{T^*_f\U}\,\cap \,\im dt\right) = (n+2)+(n+1) - 2(n+1) = 1.
$$
 Observe that the restriction, $\hat\pi$, of $\pi$ to $\overline{T^*_f\U}\ \cap\ \im dt$ yields an analytic isomorphism to its image $\pi\big(\overline{T^*_f\U}\ \cap\ \im dt\big)$, with inverse $\hat\pi^{-1}:\pi\big(\overline{T^*_f\U}\ \cap\ \im dt\big)\rightarrow \overline{T^*_f\U}\ \cap\ \im dt$ given by $\hat\pi^{-1}(x)=(x, d_xt)$.

\bigskip

We can now define our mild generalization of the relative polar curve, introduced in \cite{enrichpolar}:

\smallskip

\begin{defn}\label{def:relpolar} The {\bf relative polar set} is 
$$\big|\Gamma_{f,t}\big| =\pi\left(\overline{T^*_f\U}\ \cap\ \im dt\right).
$$

If $C$ is a $1$-dimensional irreducible component of $\big|\Gamma_{f,t}\big|$, then $\overline{T^*_f\U}$ and $\im dt$ must intersect properly over $C$ (i.e., along $\hat\pi^{-1}(C)$), and we define the {\bf multiplicity $m_C$ of $C$ in $\Gamma_{f,t}$} to be the intersection multiplicity of the cycles $\overline{T^*_f\U}$ and $\im dt$ over $C$.

If $\big|\Gamma_{f,t}|$ is purely 1-dimensional, then we define the {\bf relative polar curve} (cycle) to be  the proper push-forward of the intersection product of the cycles $\overline{T^*_f\U}$ and $\im dt$:
$$\Gamma^1_{f,t}:=\pi_*\left(\left[\overline{T^*_f\U}\right]\cdot [\im d t]\right)=\sum_Cm_C[C],$$
where the sum is over the irreducible components $C$ of $\big|\Gamma_{f,t}|$.
\end{defn}

\smallskip

\begin{rem}\label{rem:polarprops}  If $t$ is generic enough (see the next proposition) so that $\dim_\0 \big|\Gamma_{f,t}\big|\cap V(f)\leq 0$, then, near the origin,
$$
\big|\Gamma_{f,t}\big| = \overline{\Sigma(f,t) -\Sigma f},
$$
\medskip
\noindent and so  $\big|\Gamma_{f,t}\big| $ is the classical relative polar curve of Hamm, L\^e, and Teissier. 

However, frequently in these classic works, $t$ was also required to be generic enough so that the cycle $\Gamma^1_{f,t}$ was {\it reduced}, i.e., each component occurs with multiplicity $1$ in the cycle. This is equivalent to requiring that $\overline{T^*_f\U}$ transversely intersects $\im dt$ in $T^*\U$. Our results do not require $t$ to be so generic.

 In fact, with our definition, there is no need for $t$ to be a linear form; $t$ could be replaced by an arbitrary complex analytic function $g:\U\rightarrow \C$, which may have critical points. Then, the ``genericity'' that we need is still essentially the same; we require that $\dim_\0 V(g)\cap \big|\Gamma_{f,g}\big|\leq 0$.

\end{rem}

\medskip

We now prove many properties possessed by the relative polar set and curve.

\begin{prop}\label{prop:polarprops} Let $(t, z_1, \dots, z_n)$ be coordinates on $\U$. The relative polar set and curve have the following properties:

\begin{enumerate}

\item The dimension of every component of $\big|\Gamma_{f,t}\big|$ must be at least $1$, i.e., there are no isolated points in $\big|\Gamma_{f,t}\big|$.

\medskip

\item Suppose that $p\not\in\Sigma f$. Then, near $p$, 
$$\big|\Gamma_{f,t}\big| = V\left(\frac{\partial f}{\partial z_1}, \dots, \frac{\partial f}{\partial z_n}\right)=\Sigma(f,t).
$$
Thus, 
$$
\Sigma(f,t) = \Sigma f\cup \big|\Gamma_{f,t}\big|.
$$

\medskip

\item Furthermore, if $C$ is a $1$-dimensional component of $\big|\Gamma_{f,t}\big|$, and $C\not\subseteq \Sigma f$, then the multiplicity of $C$ in $\Gamma_{f,t}$, $m_C$ is precisely the multiplicity of the scheme 
$$
V\left(\frac{\partial f}{\partial z_1}, \dots, \frac{\partial f}{\partial z_n}\right)
$$
in an open neighborhood of a generic point $p$ on $C$, which equals the Milnor number $\mu_p(f_{|_{V(t-t(p))}}-f(p))$.

\medskip

\item There exists an open neighborhood $\W$ of the origin such that
$$
\W\cap \big|\Gamma_{f,t}\big|\cap V(t)=\W\cap\big|\Gamma_{f,t}\big|\cap V(f).
$$
In particular, $\dim_\0 \big|\Gamma_{f,t}\big|\cap V(t)\leq 0$ if and only if $\dim_\0 \big|\Gamma_{f,t}\big|\cap V(f)\leq 0$ (where the strict inequalities means that $\0\not\in \big|\Gamma_{f,t}\big|$, i.e., $\big|\Gamma_{f,t}\big|$ is empty at $\0$).

\medskip

\item For generic $t$, $\dim_\0 \big|\Gamma_{f,t}\big|\cap V(f)\leq 0$. 

Specifically, if $\strat$ is an $a_f$ stratification of $V(f)$ and $t$ is such that there exists an open neighborhood $\W$ of the origin such that, in $\W-\{\0\}$, $V(t)$ transversely intersects all $S\in\strat$, then $\dim_\0 \big|\Gamma_{f,t}\big|\cap V(f)\leq 0$.

In particular, if $\dim_\0\Sigma(f_{|_{V(t)}})=0$, then $\dim_\0 \big|\Gamma_{f,t}\big|\cap V(t)\leq 0$, which implies that $\dim_\0 \big|\Gamma_{f,t}\big|\cap V(f)\leq 0$.

\end{enumerate}

\end{prop}
\begin{proof} While all of these properties can be concluded from results in  \cite{enrichpolar}, the proofs and statements there are far more abstract than we need here. Hence, we will give more basic proofs which apply to our current setting.

\bigskip

\noindent{\bf Item (1)}: This is trivial. The dimension of every component of  $\overline{T^*_f\U}\ \cap\ \im dt$ is at least 
$$
\dim \overline{T^*_f\U}+\dim(\im dt) - \dim T^*\U=(n+2)+(n+1)-2(n+1) =1,
$$
and $\pi$ induces an isomorphism from this intersection to $|\Gamma_{f,t}|$.

\bigskip 

\noindent {\bf Items (2) and (3)}: We use $(w_0, w_1, \dots, w_n)$ for our cotangent coordinates and it is notationally  convenient to let $z_0:=t$, so that we have 
$$
(z_0, z_1, \dots, z_n, w_0, w_1, \dots, w_n)
$$
for coordinates on $T^*\U$. Note that, in these coordinates $\im dt=\im dz_0$ is equal to  $$V(w_0-1, w_1, \dots, w_n).$$
 Suppose that $p\not\in\Sigma f$. Then, in a neighborhood of $(p, d_pt)$, $w_0\neq 0$ and $\overline{T^*_f\U}=T^*_f\U$ is given, as a scheme, by
$$
\overline{T^*_f\U}=V\left(w_i\frac{\partial f}{\partial z_j}-w_j\frac{\partial f}{\partial z_i} \right)_{0\leq i,j\leq n} \ = \ V\left(w_0\frac{\partial f}{\partial z_j}-w_j\frac{\partial f}{\partial z_0} \right)_{1\leq j\leq n}
$$
Therefore, in a neighborhood of $p$,
$$
\pi\left(\overline{T^*_f\U}\cap \im dt\right)=\pi\left(V\left(w_0\frac{\partial f}{\partial z_j}-w_j\frac{\partial f}{\partial z_0} \right)_{1\leq j\leq n}\cap  \ V(w_0-1, w_1, \dots, w_n)=
\right)
$$
$$
\pi\left(V\left(\frac{\partial f}{\partial z_1}, \dots, \frac{\partial f}{\partial z_n}, w_0-1, w_1, \dots, w_n\right)\right)= V\left(\frac{\partial f}{\partial z_1}, \dots, \frac{\partial f}{\partial z_n}\right).
$$
This proves (2).

Now, if $C'$ is a $1$-dimensional component of $\overline{T^*_f\U}\cap\im dt$, then, at generic points of $C'$, 
$$
\left\{w_0\frac{\partial f}{\partial z_j}-w_j\frac{\partial f}{\partial z_0} \right\}_{1\leq j\leq n}, w_0-1, w_1, \dots, w_n
$$
is a regular sequence, and so the multiplicity of $C'$ in $\overline{T^*_f\U}\cdot \im dt$ is the multiplicity of the scheme
$$
V\left(\frac{\partial f}{\partial z_1}, \dots, \frac{\partial f}{\partial z_n}, w_0-1, w_1, \dots, w_n\right)
$$
along $C'$. The conclusion of (3) follows immediately.

\bigskip

\noindent {\bf Item (4)}: 

\medskip

Suppose this is false. Then either $\0\in\overline{\big|\Gamma_{f,t}\big|\cap V(t)- V(f)}$ or $\0\in\overline{\big|\Gamma_{f,t}\big|\cap V(f)- V(t)}$.

\medskip

First, we fix an $a_f$ stratification $\strat$  of $V(f)$ with connected strata; such a stratification exists by \cite{hironakastratflat}. As we used earlier in \propref{prop:contain}, the conormal characterization of an $a_f$ stratification is that, for all $S\in\strat$, 
$$
\overline{T^*_f\U}\cap\pi^{-1}(S)\subseteq T^*_S\U.
$$

Now, suppose that  $\0\in\overline{\big|\Gamma_{f,t}\big|\cap V(t)- V(f)}$. Let $\alpha(u)$ be a complex analytic curve such that $\alpha(0)=\0$ and, for $u\neq0$, $\alpha(u)\in \big|\Gamma_{f,t}\big|\cap V(t)-V(f)$. For $0<|u|\ll 1$, $\alpha(u)\not\in V(f)$ and so $\alpha(u)\not\in \Sigma f$. By Item (2), for all $u$, 
$$\alpha(u)\in V\left(\frac{\partial f}{\partial z_1}, \dots, \frac{\partial f}{\partial z_n}\right).
$$
Thus, by the Chain Rule, and using that $\alpha(u)\in V(t)$,
$$
\big(f(\alpha(u))\big)' = (t_{|_{\alpha(u)}})'\,\frac{\partial f}{\partial t}_{\big |_{\alpha(u)}}\equiv 0.
$$
Hence, $f(\alpha(u))$ is a constant, and that constant must be $0$; this contradicts that $\alpha(u)\not\in V(f)$ for $u\neq 0$.

\medskip

Suppose instead that $\0\in\overline{\big|\Gamma_{f,t}\big|\cap V(f)- V(t)}$.  Let $\alpha(u)$ be a complex analytic curve such that $\alpha(0)=\0$ and, for $u\neq0$, $\alpha(u)\in \big|\Gamma_{f,t}\big|\cap V(f)-V(t)$. Then, there exists a unique stratum $S$ in our $a_f$ stratification such that, for $0<|u|\ll 1$, $\alpha(u)\in S$. Since $\alpha(u)\in \big|\Gamma_{f,t}\big|$, 
$$(\alpha(u), d_{\alpha(u)}t)\in \overline{T^*_f\U}\cap \pi^{-1}(S)\subseteq T^*_S\U.
$$
As $\alpha'(u)\in T_{\alpha(u)}S$, we conclude that $d_{\alpha(u)}t(\alpha'(u))=(t_{|_{\alpha(u)}})'\equiv 0$. Thus, $t_{|_{\alpha(u)}}$ must be a constant, and that constant must be $0$; this contradicts that $\alpha(u)\not\in V(t)$ for $u\neq 0$.

\medskip

\noindent {\bf Item (5)}: Finally, we will show that, for a generic linear form $t$, $\dim_\0 \big|\Gamma_{f,t}\big|\cap V(f)\leq 0$. By generic, we mean that its projective class is in an open dense subset of the projectivized dual of $\C^{n+1}$. The argument that we give is due to Hamm and L\^e in \cite{hammlezariski}.

Recall that we fixed an $a_f$ stratification $\strat$ for $V(f)$ in Item (4). For each stratum $S\in\strat$, $\overline{T^*_S\U}$ is irreducible and conic of dimension $n+1$. Writing $\Proj(\overline{T^*_S\U})$ for its projectivization in the cotangent directions, we have that $\Proj (\overline{T^*_S\U})$ is an irreducible $n$-dimensional subvariety 
$$\Proj(\overline{T^*_S\U})\subseteq \Proj(T^*\U)\cong \U\times \Proj^n.
$$
Therefore, if $S\neq \{\0\}$, then the fiber over the origin $\Proj(\overline{T^*_S\U})_\0$ is a subvariety of $\{\0\}\times \Proj^n\cong\Proj^n$ of dimension at most $n-1$. Consider the open dense subset $\Omega$ of $\Proj^n$ given by
$$
\Omega:=\Proj^n-\bigcup_{S\neq\{\0\}}\Proj(\overline{T^*_S\U})_\0.
$$
Let $t$ be such that its projective class $[d_\0t]$ is in $\Omega$, i.e., $t$ is such that
$$
d_\0t\not\in \bigcup_{S\neq\{\0\}}\left(\overline{T^*_S\U}\cap \pi^{-1}(\0)\right);
$$ 
this is equivalent to selecting $t$ so that there exists an open neighborhood $\W$ such that $V(t)$ transversely intersects all $S\in\strat$ inside $\W-\{\0\}$ (where we use that stratified critical points of $t$ near $\0$ must occur inside $V(t)$).

Thus, 
$$
\pi^{-1}(\W)\cap \im dt\cap \bigcup_{S\neq\{\0\}}\overline{T^*_S\U}=\emptyset.
$$

Now, using that $\strat$ is an $a_f$ stratification, we find
$$
\pi^{-1}(\W)\cap\overline{T^*_f\U}\cap (f\circ\pi)^{-1}(0)\cap \im dt   \ =  \ \pi^{-1}(\W)\cap \overline{T^*_f\U}\cap \pi^{-1}\big(\bigcup_{S\in\strat}S\big)\cap \im dt \ \subseteq
$$
$$
\pi^{-1}(\W)\cap\big(\bigcup_{S\in\strat}T^*_S\U\big)\cap \im dt = \pi^{-1}(\W)\cap T^*_{\{\0\}}\U\cap \im dt.
$$

Therefore, $\W\cap |\Gamma_{f,t}|\cap V(f)\subseteq \{\0\}$, i.e., $\dim_\0 \big|\Gamma_{f,t}\big|\cap V(f)\leq 0$.

\smallskip

The last statement of this item follows from the $a_f$ statement, but is trivial to prove independently.
$$
\Sigma(f_{|_{V(t)}}) = V\left(\frac{\partial f}{\partial z_1}, \dots, \frac{\partial f}{\partial z_n}, t\right),
$$
which, by Item (2), is equal to
$$
\left(\Sigma f\cup \big|\Gamma_{f,t}\big|\right)\cap V(t)
$$
If the dimension of this at the origin is at most $0$, then certainly $\dim_\0 \big|\Gamma_{f,t}\big|\cap V(t)\leq 0$.
\end{proof}

\medskip

\begin{rem} Given the previous proposition, the definition of the relative polar curve given in \defref{def:relpolar} may seem needlessly convoluted. Why not just define $\Gamma_{f,t}$ to be $\overline{\Sigma(f,t)-\Sigma f}$ and, if this is $1$-dimensional, give it the cycle structure in which the coefficient of a component $C$ is the Milnor number of $f-f(p)$, restricted to a transverse hyperplane slice to $C$ at a generic point $p\in C$?

The answer is that, in the next two sections, we will use many results from \cite{enrichpolar}, where we saw that the property that we needed for $t$ to possess was precisely that $(\0, d_\0t)$ be an isolated point of
$$
\overline{T^*_f\U}\cap (f\circ\pi)^{-1}(0)\cap \im dt.
$$
Thus, defining the relative polar curve as we did in \defref{def:relpolar} allows us to state the property which we need satisfied in a very simple way; we require $\dim_\0 \big|\Gamma_{f,t}\big|\cap V(f)\leq 0$.
\end{rem}

\medskip

\section{Generalizations of Classical results}

The only theorem in this section contains, on the level of cohomology,  generalizations of L\^e's main result in \cite{leattach}, a result of Siersma in \cite{siersmalink}, and a result of L\^e and Perron in \cite{leperron}. These are generalizations in the sense that the hypothesis needed is significantly weaker than those used in the earlier theorems.

\smallskip

We let $f_0:=f_{|_{V(t)}}$, and we let $F_{f, \0}$ and $F_{f_0, \0}$ denote, respectively, the Milnor fibers of $f$ and $f_0$ at the origin.  Also let $B^\circ_\epsilon(\0)$ denote the open ball of radius $\epsilon>0$, centered at $\0$.

\smallskip

Recall that, from \remref{rem:polarprops}, $\dim_\0 \big|\Gamma_{f,t}\big|\cap V(t)\leq 0$ if and only if $\dim_\0 \big|\Gamma_{f,t}\big|\cap V(f)\leq 0$.

\medskip

\begin{thm}\label{thm:vangen} Suppose that $\dim_\0 \big|\Gamma_{f,t}\big|\cap V(f)\leq 0$ (or, equivalently, $\dim_\0 \big|\Gamma_{f,t}\big|\cap V(t)\leq 0$).  

Then, for $0<|v|\ll |a|\ll \epsilon\ll 1$,
\begin{enumerate}
\item For all $k$, there are isomorphisms
$$
H^k(B^\circ_\epsilon(\0)\cap f^{-1}(v), \ B^\circ_\epsilon(\0)\cap t^{-1}(a)\cap f^{-1}(v);\,\Z)\cong $$
$$
H^k(B^\circ_\epsilon(\0)\cap f^{-1}(v), \ B^\circ_\epsilon(\0)\cap V(t)\cap f^{-1}(v);\,\Z)=
$$
$$
H^k(F_{f, \0}, \ F_{f_0,\0};\,\Z)\cong\begin{cases} \Z^\tau, &\textnormal{ if } k=n,\\
0, &\textnormal{ if } k\neq n,
\end{cases}
$$
where $\tau:= (\Gamma^1_{f,t}\cdot V(f))_\0$.

\bigskip

\item The complex link of $V(f)$ at the origin (with respect to $t$), $$\cL_{V(f), \0}:=B^\circ_\epsilon(\0)\cap V(f)\cap t^{-1}(a),$$ has reduced cohomology given by
$$
\widetilde H^k(\cL_{V(f), \0}; \Z) \cong \begin{cases} \Z^\gamma, &\textnormal{ if } k=n-1,\\
0, &\textnormal{ if } k\neq n-1,
\end{cases}
$$
where $\gamma:=(\Gamma^1_{f,t}\cdot V(t))_\0$.

 \end{enumerate}
\end{thm}
\begin{proof}  In the statement of the theorem, we have suppressed any use of the derived category, vanishing cycles, and nearby cycles. However, we will necessarily need them to translate the results \cite{enrichpolar} into the forms that appear in \thmref{thm:vangen}.

First, 
$$
H^k(B^\circ_\epsilon(\0)\cap f^{-1}(v), \ B^\circ_\epsilon(\0)\cap t^{-1}(a)\cap f^{-1}(v);\,\Z)\cong H^{k-n}(\phi_g[-1]\psi_f[-1]\Z^\bullet_\U[n+1])_\0.
$$
 Now, Theorem 3.14 of \cite{enrichpolar} implies that 
$$
H^k(B^\circ_\epsilon(\0)\cap f^{-1}(v), \ B^\circ_\epsilon(\0)\cap t^{-1}(a)\cap f^{-1}(v);\,\Z)\cong \begin{cases} \Z^\tau, &\textnormal{ if } k=n,\\
0, &\textnormal{ if } k\neq n.
\end{cases}
$$
That this is also isomorphic to $H^k(F_{f, \0}, \ F_{f_0,\0};\,\Z)$ is immediate from Item 1 of Corollary 4.6 of \cite{enrichpolar}.

Item 2 of Corollary 4.6 of \cite{enrichpolar} tells us that
$$
H^{n+k}(F_{t, \0}, F_{t_{|_{V(f)}},\0}; \Z)
$$
is non-zero only if $k=0$ and, if $k=0$, is isomorphic to $\Z^\gamma$. As $F_{t, \0}$ is contractible, Item 2 of \thmref{thm:vangen} follows from the long exact sequence of the pair $(F_{t, \0}, F_{t_{|_{V(f)}},\0})$.
\end{proof}

\medskip

\begin{rem} It is important to note that we have said that \thmref{thm:vangen} generalizes earlier works {\bf on the level of cohomology}. In fact, these earlier works have stronger conclusions.

The first isomorphism of Item 1 of \thmref{thm:vangen} is the cohomological version of an isotopy result proved in L\^e and Perron in \cite{leperron} (their ambient space is $\C^3$, but that is irrelevant to their proof). In \cite{leattach}, L\^e actually proves that $F_{f,\0}$ is obtained, up to homotopy, from $F_{f_0, \0}$ by attaching $\tau$ $n$-cells. In \cite{siersmalink}, Siersma proves that $\cL_{V(f), \0}$ has the homotopy-type of a bouquet of $\gamma$ $(n-1)$-spheres.

\end{rem}

\medskip

\section{IPA-deformations}

In Section 3, we avoided referring to hypercohomology and vanishing cycles in our statements of results. In this section, that is not possible. However, we shall give some examples which will, hopefully, make the results more accessible.

\begin{defn}\label{def:ipadef}  Let $\W$ be an open neighborhood of the origin in $\C^n$, and suppose we have a complex analytic function $f_0:(\W,\0)\rightarrow(\C,0)$. Then, a {\bf deformation of $f_0$ with isolated polar activity at $\0$, or an IPA-deformation of $f_0$, with parameter $t$}, is a complex analytic $f:\D^\circ\times\W\rightarrow\C$, where $\D^\circ$ is an open disk around the origin in $\C$, such that, if $t$ denotes the projection onto $\D^\circ$ and we identify $\W$ with $\{0\}\times\W$, then $f_0=f_{|_{V(t)}}$ and $\dim_\0 \big|\Gamma_{f, t}\big|\cap V(t)\leq 0$ (or, equivalently, $\dim_\0 \big|\Gamma_{f, t}\big|\cap V(f)\leq 0$).

A {\bf null IPA-deformation} is an IPA-deformation for which $\0\not\in \big|\Gamma_{f, t}\big|$.
\end{defn}

\bigskip

{\bf Throughout the remainder of this paper, we assume that $f$ is an IPA-deformation of $f_0:(\W,\0)\rightarrow(\C,0)$ at $\0$, where $\W$ is an open neighborhood of the origin in $\C^n$}.

\bigskip

The following proposition tells us that the critical locus is well-behaved in IPA-deformations.

\begin{prop}\label{prop:dimbehave} The following cases can occur:
\begin{itemize}
\item If $\0\not\in \Sigma f$, then, in a neighborhood of the origin, $\Sigma(f_0)=|\Gamma_{f,t}|\cap V(t)$, and so either $\0\not\in\Sigma{f_0}$ or $\dim_\0\Sigma(f_0)=0$.

\medskip

\item If $\0\in\Sigma f$, then in a neighborhood of the origin, $\Sigma(f_0)=\Sigma f\cap V(t)$ and, if $\dim_\0\Sigma f\geq 1$, then $\dim_\0\Sigma f\cap V(t)=\dim_\0\Sigma f-1$.
\end{itemize}
\end{prop}
\begin{proof} We use $(t, z_1, \dots, z_n)$ for coordinates on $\U$. Then,
$$
\Sigma(f_0) = V\left(\frac{\partial f}{\partial z_1}, \dots, \frac{\partial f}{\partial z_n}, t\right) = \left(\Sigma f\cup |\Gamma_{f, t}|\right)\cap V(t) = \left(\Sigma f\cap V(t)\right)\cup \left(|\Gamma_{f, t}|\cap V(t)\right).
$$
As we are assuming that $\dim_\0|\Gamma_{f, t}|\cap V(t)\leq 0$, all of the conclusions follow, with the exception of the final dimension claim. 

The only way that we can have $\dim_\0\Sigma f\cap V(t)\neq\dim_\0\Sigma f-1$ is if $V(t)$ contains an irreducible component $Y$ of $\Sigma f$ which contains the origin. However, \propref{prop:contain} tells us that $\overline{T^*_{Y_{\operatorname{reg}}}\U}\subseteq \overline{T^*_f\U}$. Furthermore, $Y\subseteq V(t)$ implies that $\pi\left(\overline{T^*_{Y_{\operatorname{reg}}}\U}\cap \im dt\right)=Y$. But this implies that $Y\subseteq |\Gamma_{f,t}|\cap V(t)$. This would contradict that the dimension of $Y$ is at least $1$ and that $f$ is an IPA-deformation.
\end{proof}
\medskip

Below, we will use the Milnor fiber, $F_{t, \0}$, of $t$ at the origin, and consider the space 
$$F_{t, \0}\cap \Sigma f = B^\circ_\epsilon(\0)\cap \Sigma f\cap t^{-1}(a),
$$
where $0< |a|\ll \epsilon\ll 1$. This should be thought of as the ``complex link of $\Sigma f$ at $\0$'' (with respect to $t$). 

The vanishing cycles along $f$ are denoted by $\phi_f$, and $\hyp$ denotes hypercohomology with respect to a complex of sheaves.

\begin{thm}\label{thm:ipa} There are isomorphisms:

\smallskip

\noindent For $k\neq n-1$,
$$
\widetilde H^{k}(F_{f_0,\0}; \Z)\ \cong\  \hyp^k(F_{t,\0}\cap\Sigma f; \phi_f\Z_\U^\bullet),
$$
and 
$$
\widetilde H^{n-1}(F_{f_0,\0}; \Z) \ \cong \ \Z^\gamma\ \oplus\ \hyp^{n-1}(F_{t,\0}\cap\Sigma f; \phi_f\Z_\U^\bullet),
$$
\medskip
\noindent where $\gamma:=(\Gamma^1_{f,t}\cdot V(t))_\0$.

In particular, $\operatorname{rank} \widetilde H^{n-1}(F_{f_0,\0}; \Z)\geq \gamma$.
\end{thm}
\begin{proof} This is precisely Item 3 of Corollary 4.6 of  \cite{enrichpolar}, with $\Fdot=\Z^\bullet_\U[n+1]$, noting that 
$$\psi_t[-1]\Z^\bullet_\U[n+1]\cong\Z^\bullet_{|_{V(t)}}[n],
$$
since $\phi_t[-1]\Z^\bullet_\U[n+1]=0$ as $t$ has no critical points.
\end{proof}

\medskip

As an immediate corollary, we have:

\begin{cor} Suppose that $H^{n-1}(F_{f_0, \0};\Z)=0$. Then $f$ is a null IPA-deformation.
\end{cor}

\medskip

We wish now to give two examples. 

\medskip

\begin{exm} Suppose $\dim_\0\Sigma(f_0)=0$. 

\smallskip

If $\dim_\0\Sigma f\leq 0$, then $F_{t,\0}\cap \Sigma f=\emptyset$ and $\hyp^k(F_{t,\0}\cap\Sigma f; \phi_f\Z_\U^\bullet)=0$ for all $k$.

\smallskip

  If $\dim_\0\Sigma f\geq 1$, then 
  
  \begin{itemize}
  \item $F_{t,\0}\cap \Sigma f$ consists of a finite number of points $\{p_1, \dots, p_m\}$, 
  
  \smallskip
  
  \item  $\hyp^k(F_{t,\0}\cap\Sigma f; \phi_f\Z_\U^\bullet)=0$ for  $k\neq n-1$, and
  
  \smallskip
  
  \item $\hyp^{n-1}(F_{t,\0}\cap\Sigma f; \phi_f\Z_\U^\bullet)=\bigoplus_i\Z^{\mu_{p_i}(f_{|_{V(t-t(p_i))}})}$.
  
  \end{itemize}
  
  \smallskip

Thus, we arrive at the well-known, easy result that
$$
\mu_\0(f_0)= (\Gamma^1_{f,t}\cdot V(t))_\0 \ +\sum_{p\in F_{t,\0}\cap\Sigma f} \mu_p(f_{|_{V(t-t(p))}}).
$$

\end{exm}

\medskip

\begin{exm}\label{exm:onedim} Suppose $\dim_\0\Sigma(f_0)=1$. Then, \propref{prop:dimbehave} tells us that $\dim_0\Sigma f = 2$ and, hence, $F_{t,\0}\cap \Sigma f$ is $1$-dimensional. Calculating $\hyp^k(F_{t,\0}\cap\Sigma f; \phi_f\Z_\U^\bullet)$ in this case is generally highly non-trivial. But, in this example, we will look at a special case.

\smallskip

First, we need to discuss the cohomology of the Milnor fiber at the origin of $g(x,y,s):=y^2-x^b-sx^a$, where $b>a\geq 2$. This does not use IPA-deformations.

By the Sebastiani-Thom result \cite{sebthom}, the Milnor fiber $F_{g,\0}$ is, up to homotopy, the suspension of the Milnor fiber of $h(x,s):= -x^b-sx^a=-x^a(x^{b-a}+s)$. After an analytic change of coordinates at the origin, letting $\hat s:=x^{b-a}+s$, we find that $h$ becomes $\hat h(x, \hat s)=-x^a\hat s$. This is a homogeneous polynomial, and so, by Lemma 9.4 of \cite{milnorsing}, $F_{\hat h, \0}\cong \hat h^{-1}(1)$. But this is simply the graph of the function $k:\C^*\rightarrow\C^2$ given by $k(x)=-1/x^a$, which is isomorphic to $\C^*$. Thus, we find that, regardless of the values of $a$ and $b$, $F_{g, \0}$ has the homotopy-type of $S^2$ and so, for $k\neq 2$, $\widetilde H^k(F_{g, \0};\Z)=0$ and $\widetilde H^2(F_{g, \0};\Z)\cong\Z$.

\medskip

Now consider $f_0(x,y,s):=y^2-x^b-s^mx^a$, where $b>a\geq 2$ and $m\geq 2$. Again, this is a suspension, but we can no longer perform an analytic change of coordinates on $-x^a(x^{b-a}+s^m)$ to immediately determine the homotopy-type or cohomology of $F_{f_0, \0}$. It is true that we could analyze this by considering the set where $-x^a(x^{b-a}+s^m)=1$ as an $m$-fold branched cover of $\C^*$, but, instead we will use IPA-deformations.

\medskip

We claim that $f(x,y,s,t):=y^2-x^b-s^mx^a+tx^a$ is an IPA-deformation of $f_0$ and that \thmref{thm:ipa} allows us to calculate $H^*(F_{f_0,\0};\Z)$.

\smallskip

We first want to produce an $a_f$ stratification of $V(f)$. It is trivial to verify that $\Sigma f=V(x,y)$. Now, consider the $2$-parameter family of isolated critical points given by 
$$f_{s,t}(x,y):=y^2-x^b-s^mx^a+tx^a.
$$
The partial derivatives are
$$
\frac{\partial f_{s,t}}{\partial x}=-bx^{b-1}-as^mx^{a-1}+atx^{a-1}=-x^{a-1}\big(bx^{b-a}+a(s^m-t)\big)\hskip 0.3in\textnormal{and}\hskip 0.3in \frac{\partial f_{s,t}}{\partial y}=2y,
$$
and we find that the Milnor numbers of $f_{s,t}$ at $(x,y)=(0,0)$ are given by
$$
\mu_\0(f_{s,t})=\begin{cases}a-1, \textnormal{ if } s^m-t\neq 0,\\
b-1,  \textnormal{ if } s^m-t= 0.
\end{cases}
$$

By Theorem 6.8 of \cite{lecycles}, this implies that 
$$
\strat:=\{V(f)-V(x,y), \,V(x,y)-V(s^m-t), \,V(x,y, s^m-t)\}
$$
is an $a_f$ stratification of $V(f)$. Furthermore, $V(t)$ clearly transversely intersects $V(f)-V(x,y)$ and $V(x,y)-V(s^m-t)$ and also, vacuously, transversely intersects $V(x,y,s^m-t)$ in $\C^4-\{\0\}$. Therefore, by Item (5) of \propref{prop:polarprops}, $f$ is an IPA-deformation of $f_0$.

Thus, \thmref{thm:ipa} tells us that:

\smallskip

\noindent for $k\neq 2$,
$$
\widetilde H^{k}(F_{f_0,\0}; \Z)\ \cong\  \hyp^k(F_{t,\0}\cap\Sigma f; \phi_f\Z_\U^\bullet),
$$
and 
$$
\widetilde H^2(F_{f_0,\0}; \Z) \ \cong \ \Z^\gamma\ \oplus\ \hyp^{2}(F_{t,\0}\cap\Sigma f; \phi_f\Z_\U^\bullet),
$$

\medskip

\noindent where $\gamma:=(\Gamma^1_{f,t}\cdot V(t))_\0$.

Items (2) and (3) of \propref{prop:polarprops} tell us how to calculate $\Gamma^1_{f,t}$. First, we have
$$
\Sigma f\cup |\Gamma_{f,t}| \ = \ V\left(\frac{\partial f}{\partial x}, \frac{\partial f}{\partial y}, \frac{\partial f}{\partial s}\right) = V\big(-x^{a-1}(bx^{b-a}+a(s^m-t)), \,2y, \,-ms^{m-1}x^a\big)=
$$
$$
V(x,y)\cup V\big(bx^{b-a}+a(s^m-t), \,y, \,s^{m-1}\big), 
$$
and we see that 
$$ |\Gamma_{f,t}|=V\big(bx^{b-a}+a(s^m-t), \,y, \,s^{m-1}\big).$$
However, above, we were careful to preserve the cycle structure on $|\Gamma_{f,t}|$ in our calculation. 
Thus, we find
$$
\Gamma^1_{f,t}=\left[V\big(bx^{b-a}+a(s^m-t), \,y, \,s^{m-1}\big)\right]=(m-1)V(bx^{b-a}-at, y, s)
$$
and
$$
(\Gamma^1_{f,t}\cdot V(t))_\0= \big((m-1)V(bx^{b-a}-at, y, s)\cdot V(t)\big)_\0= (m-1)(b-a).
$$

\bigskip

It remains for us to calculate $\hyp^k(F_{t,\0}\cap\Sigma f; \phi_f\Z_\U^\bullet)$.

\medskip

For $0<|a|\ll \epsilon\ll 1$, 
$$F_{t,\0}\cap\Sigma f=B^\circ_\epsilon(\0)\cap \Sigma f\cap V(t-a)
$$ is an open disk $D$ containing the origin in the copy of the $s$-plane where $x=0$, $y=0$, and $t=a$. By our earlier calculation of $\mu_\0(f_{s,t})$, we know that the restriction to $D$ of $\phi_f\Z_\U^\bullet$ is locally constant, with stalk cohomology which is non-zero only in degree $1$, on $D-\{p_1, \dots, p_m\}$, where the $p_i$'s are the $m$ distinct $m$-th roots of $a$; the stalk cohomology of this local system in degree $1$ is isomorphic to $\Z^{a-1}$. In addition, our earlier calculation of the cohomology of the Milnor fiber of $g(x,y,s):=y^2-x^b-sx^a$ at $\0$ tells us that, at each $p_i$, the stalk cohomology at $p_i$ of the restriction to $D$ of $\phi_f\Z_\U^\bullet$ is zero in all degrees other than degree $2$, where the cohomology is isomorphic to $\Z$.

Now, an easy induction on $m$, using the Mayer-Vietoris long exact sequence for hypercohomology, tells us that, for $k\neq 2$, $\hyp^k(F_{t,\0}\cap\Sigma f; \phi_f\Z_\U^\bullet)=0$  and $\hyp^2(F_{t,\0}\cap\Sigma f; \phi_f\Z_\U^\bullet)\cong \Z^{(m-1)a+1}$. Therefore, we conclude that, for $k\neq 2$, $\widetilde H^{k}(F_{f_0,\0}; \Z)=0$,
and 
$$
\widetilde H^2(F_{f_0,\0}; \Z) \ \cong \ \Z^{(m-1)(b-a)}\oplus \Z^{(m-1)a+1} \ \cong \ \Z^{(m-1)b+1}.
$$

\medskip

Note that, when $m=1$, we obtain our previous result for $F_{g, \0}$. Also note that, when $a=2$, one obtains from \cite{siersmaisoline}, and the calculation of the Euler characteristic via \cite{leattach}, that $F_{f_0, \0}$ has the homotopy-type of a bouquet of $(m-1)b+1$ two-spheres.

\end{exm}

\section{Questions, Comments, and Future Directions}

\begin{itemize}

\item The isomorphisms in \thmref{thm:ipa} may be thought of as refinements of the formulas, given in Proposition 1.21 of \cite{lecycles}, for the L\^e numbers of a function restricted to hyperplane slice.

\bigskip

\item \exref{exm:onedim} is very special. In general, it is unclear precisely when one can explicitly calculate $\hyp^k(F_{t,\0}\cap\Sigma f; \phi_f\Z_\U^\bullet)$ or even obtain better bounds than are currently known.

\bigskip

\item Perhaps the next ``easiest'' case where one can specialize the results of \cite{enrichpolar} is the case where  $X$ is a local complete intersection (LCI) and we want to deform $f_0:X\rightarrow\C$ via $f:\widetilde X\rightarrow\C$, where $\widetilde X$ is again an LCI. The point is that $\widetilde X$ being a purely $d$-dimensional LCI implies $\Z^\bullet_{\widetilde X}[d]$ is a perverse sheaf, which simplifies the results of \cite{enrichpolar}. 

However, the LCI case is still much more complicated than the affine case. First, because the vanishing cycles $\phi_t\Z^\bullet_{\widetilde X}[d]$ need not be zero and, second, because the data that we would need about $\widetilde X$ -- before considering $f$ -- is the characteristic cycle $CC(\widetilde X)$. This is highly non-trivial data, which is not easy to calculate given the defining functions for $\widetilde X$.
\end{itemize}

\medskip

\bigskip

\printbibliography
%\bibliographystyle{plain}
%\bibliographystyle{amsalpha}
%\bibliography{Masseybib}
%\printindex
\end{document}